\documentclass[a4paper]{amsart}

\usepackage{graphicx}
\usepackage[all,cmtip]{xy}
\usepackage{amsmath, amssymb, amsthm}
\usepackage{amsrefs}
\usepackage{stmaryrd}
\usepackage{tikz}
\usepackage{tkz-euclide}
\usetikzlibrary{arrows}
\usetikzlibrary{calc}
\usetikzlibrary{decorations.markings}
\usetikzlibrary{decorations.pathmorphing}
\usetikzlibrary{positioning}

\newcommand{\setuphypplane}{
	\tkzDefPoint(0,0){O}
	\tkzDefPoint(1,0){A}
	\tkzClipCircle(O,A)
}

\newcommand{\hgline}[3]{
	\begin{scope}
		\tkzDefPoint(#1:1.0000001){z1}
		\tkzDefPoint(#2:1.0000001){z2}
		\tkzDrawCircle[#3,orthogonal through=z1 and z2](O,A)
	\end{scope}
}

\newcommand{\lamination}[2]{
	\foreach \a/\b in {#2}
	{
		\hgline{\a}{\b}{very thin,color=#1}
	}
}

\newcommand{\drawBoundaryCircle}{
	\tkzDrawCircle[ultra thick](O,A)
}
 
\tikzset{vertex/.style={circle,draw,fill=black,inner sep=0pt, minimum
width=2pt}}

\newtheorem{theorem}{Theorem}[section]
\newtheorem{lemma}[theorem]{Lemma}
\newtheorem{proposition}[theorem]{Proposition}
\newtheorem{corollary}[theorem]{Corollary}

\newtheorem*{theorem*}{Theorem}

\theoremstyle{definition}
\newtheorem{definition}[theorem]{Definition}

\theoremstyle{remark}
\newtheorem{remark}[theorem]{Remark}

\numberwithin{equation}{section}

\DeclareMathOperator{\len}{length}
\DeclareMathOperator{\Fix}{Fix}

\newcommand{\sO}{\mathsf{O}}
\newcommand{\sD}{\mathsf{D}}

\newcommand{\ebox}[1]{\llbracket #1 \rrbracket}
\newcommand{\id}{\mathrm{id}}
\newcommand{\piS}{\pi_1(\Sigma)}

\begin{document}

\title[Compatible Real Trees]{Length Function Compatibility for Group Actions
on Real Trees}

\author{Edgar A. Bering IV}
\address{Faculty of Mathematics\\
Technion---Israel Institute of Technology\\
3200003, Haifa\\
Israel}
\curraddr{}
\email{bering@campus.technion.ac.il}
\thanks{This article is based upon the author's thesis submitted in partial fulfillment
of the requirements for a Ph. D. at the University of Illinois at Chicago}



\begin{abstract}
Let $G$ be a finitely generated group.  Given two length functions $\ell$ and
$m$ of irreducible $G$ actions on real trees $A$ and $B$, when is the
point-wise sum $\ell + m$ again the length function of an irreducible $G$
action on a real tree? Guirardel and Levitt showed that additivity is
equivalent to the existence of a common refinement of $A$ and $B$, this
equivalence is established using Guirardel's core. Moreover, in this case the
sum $\ell + m$ is the length function of the common refinement of $A$ and $B$
given explicitly by the Guirardel core. The core can be difficult to compute
in general. Behrstock, Bestvina, and Clay give an algorithm for computing the
core for free group actions on simplicial trees. In this article we give a
geometric characterization of existence of a common refinement that generalizes
the criterion underlying Behrstock, Bestvina, and Clay's algorithm, as well as
two equivalent characterizations in terms of the associated translation length
functions.
\end{abstract}

\maketitle

\section{Introduction}

Suppose $\lambda$ and $\mu$ are sets of disjoint simple closed geodesics on a
closed hyperbolic surface $\Sigma$. Further suppose
the geometric intersection number $i(\lambda,\mu) = 0$ so that the union
$\lambda\cup \mu$ is again a set of disjoint simple closed geodesics. The lifts
of $\lambda$ to the universal cover, $\tilde{\lambda}$, describe a simplicial
tree $A$ with $\piS$ action. The vertices of $A$ are the connected
components of $\mathbb{H}^2\setminus\tilde{\lambda}$, and vertices $X$ and $Y$
are joined by an edge if $\bar{X}\cap\bar{Y} = \gamma$ is a geodesic in
$\tilde{\lambda}$. Similarly the lifts of $\mu$ describe a tree $B$ with
$\piS$ action, and $\lambda\cup\mu$ a tree $T$. Since $\lambda$ and
$\mu$ are disjoint, $T$ comes with equivariant surjections $T\to A$ and $T\to
B$. From the construction, these surjections have the property that every
segment $[a,b]$ is sent to the segment $[f(a),f(b)]$, we say that these
surjections \emph{preserve alignment}. Further, each tree has a metric induced
by assigning each edge length one, which allows us to define translation length
functions $\ell_A, \ell_B, \ell_T: \piS\to \mathbb{R}^{+}$.  From the
construction of $T$ it follows that the surjections to $A$ and $B$ are
Lipschitz; and we can calculate $\ell_T = \ell_A+\ell_B$.

This kind of compatibility generalizes to the setting of a finitely generated
group $G$ acting on real trees $A$ and $B$ (hereafter $G$-trees).

\begin{definition}
	A $G$-tree $T$ is a \emph{common refinement} of $G$-trees $A$ and $B$ if there
	are equivariant Lipschitz surjections $T\to A$ and $T\to B$ that
	preserve alignment.
\end{definition}

Guirardel~\cite{guirardel-core} introduces a convex core (and a notion of
intersection number) for a pair of $G$-trees, and
proves~\cite{guirardel-core}*{Theorem 6.1} that this convex core is
one-dimensional if and only if there exists a common refinement of the two
$G$-trees.

In the opening example, the compatibility of two sets of simple closed curves
also entailed a compatibility for the resulting length functions on $\piS$. One
family of actions, known as irreducible actions
(Definition~\ref{def:irreducible}), are completely characterized by their
length functions~\cites{culler-morgan,parry}. In the more general setting of
$G$-trees, Guirardel and Levitt~\cite{guirardel-levitt}*{Appendix A} prove that
two minimal, irredicuble $G$-trees have a common refinement if and only if the
sum of their length functions is again the length function of a minimal,
irreducible $G$-tree~\cite{guirardel-levitt}*{Appendix A}. Moreover, in this
case there is a common refinement with length function equal to the sum.

This article introduces three new equivalent characterizations of the existence
of a common refinement of $G$-trees. Each is motivated from the surface
setting, we defer formal definitions to Section~\ref{sec:comb-compat} after
the necessary background is introduced. 

Generalizing slightly from the opening example, consider two measured geodesic
laminations $\lambda$ and $\mu$ on a closed hyperbolic surface $\Sigma$.
There are naturally associated $\piS$-trees $A$ and $B$ dual to
$\lambda$ and $\mu$~\cite{surface-dual-tree}. To define \emph{incompatibility}, 
suppose $\lambda$ and $\mu$ have leaves that intersect transversely. This
intersection produces arcs $a\subseteq A$ and $b\subseteq B$ in the dual trees
to $\lambda$ and $\mu$ such that the Gromov boundaries of particular
complementary components of $a$ and $b$, all intersect, as
in~Figure~\ref{fig:horizon-isect}. Definition~\ref{def:comp-tree} re-interprets this boundary intersection as an
intersection condition for four particular subsets of $G$ associated to arcs in
the $G$-trees under consideration, and two $G$-trees have \emph{incompatible
horizons} if such an intersection occurs.

\begin{figure}
	\begin{center}
		\begin{tikzpicture}[scale=3,font=\small,decoration={
    markings,
    mark=at position 0.5 with {\arrow{>}}}]
	\node [red] at (-90:1.1) {$\lambda$};
	\node [blue] at (0:1.1) {$\mu$};

	\draw [red] ([shift={(-80:1.3)}]0,0) -- (-80:1.4) arc (-80:85:1.4) node
	[midway, right] {$\ebox{a}$} -- (85:1.3);
	\draw [red] ([shift={(95:1.3)}]0,0) -- (95:1.4) arc (95:250:1.4) node
	[midway, left] {$\ebox{\bar{a}}$} -- (250:1.3);
	\draw [blue] ([shift={(170:1.1)}]0,0) -- (170:1.2) arc (170:10:1.2)
	node [midway, above] {$\ebox{b}$} -- (10:1.1);
	\draw [blue] ([shift={(198:1.1)}]0,0) -- (198:1.2) arc (198:345:1.2)
	node [midway, below] {$\ebox{\bar{b}}$} --(345:1.1);

	\setuphypplane
	\lamination{red}{-21/-30,-18/-45,-15/-75,25/50,27/40,-80/85,-80/90,-80/95,-110/95,-105/95,120/150,120/140,130/140,-130/-170,-135/-170,-140/-155}
	\lamination{blue}{0/198,10/170,10/175,10/178,-15/198,-90/198,-95/215,-95/-105,-15/-50,-20/-40,-35/-40,-50/-90,30/110,40/90,40/70,110/170}
	\node [red] at (-.35,.1) {$a$} ;
	\draw [red, thick, postaction={decorate}] (-.3,.1) -- (.1,.1);

	\node [blue] at (.25,-.35) {$b$};
	\draw [blue, thick, postaction={decorate}] (.25,-.3) -- (.25,.1);

	\drawBoundaryCircle
\end{tikzpicture}
 		\caption{Intersecting measured laminations produce intersecting
		boundaries.}
		\label{fig:horizon-isect}
	\end{center}
\end{figure}

The intersection of the laminations $\lambda$ and $\mu$ is also detected by
certain pairs of elements of $\piS$. The hyperbolic structure on $\Sigma$ gives
an action of $\piS$ on the hyperbolic plane $\mathbb{H}^2$, and elements of
$\piS$ act hyperbolically. Given two elements $x,y\in \piS$, if the axes of $x$
and $y$ are separated by a set of leaves of positive measure of the lift
$\tilde{\lambda} \subseteq \mathbb{H}^2$ then the axes of $x$ and $y$ in the
dual tree $A$ will be disjoint. On the other hand, if the axes of $x$ and $y$
cross a common set of leaves of $\tilde{\mu}$ (with positive measure) then
their axes in the dual tree $B$ will overlap in an arc. The presence of a pair
of such elements detects the intersection of $\lambda$ and $\mu$, as
illustrated in~Figure~\ref{fig:lam-inc-comb}. Generalizing to $G$-trees, we say
(Definition~\ref{def:cc}) two $G$-trees have \emph{compatible combinatorics} if
there is no pair of elements $x, y\in G$ that have disjoint axes in one tree
and overlapping axes in the other. This definition is stated synthetically,
entirely in terms of length function inequalities.

\begin{figure}
	\begin{center}
		\begin{tikzpicture}[scale=2.5,font=\small,decoration={
    markings,
    mark=at position 0.5 with {\arrow{>}}}
]
	\begin{scope}
		\draw ([shift={(95:1.1)}]0,0) -- (95:1.15) arc (95:85:1.15) node
		[midway, above, align=center] {separating leaves} -- (85:1.1);
		\node [blue] at (-152:1.1) {$x$};
		\node [blue] at (-60:1.1) {$y$};
		\node [red] at (-90:1.1) {$\lambda$};
		\setuphypplane
		\lamination{red}{-21/-30,-18/-45,-15/-75,25/50,27/40,-80/85,-80/90,-80/95,-110/95,-105/95,120/150,120/140,130/140,-130/-170,-135/-170,-140/-155}
		\hgline{-152}{127}{color=blue}
		\hgline{-60}{35}{color=blue}
		\drawBoundaryCircle
	\end{scope}
	\begin{scope}[shift={(2.5,0)}]
		\draw ([shift={(10:1.1)}]0,0) -- (10:1.15) arc (10:-15:1.15) --
		(-15:1.1);
		\node [align=center] at (90:1.2){common crossed\\ leaves};
		\node [blue] at (-152:1.1) {$x$};
		\node [blue] at (-60:1.1) {$y$};
		\node [red] at (-90:1.1) {$\mu$};
		\setuphypplane
		\lamination{red}{0/198,10/170,10/175,10/178,-15/198,-90/198,-95/215,-95/-105,-15/-50,-20/-40,-35/-40,-50/-90,30/110,40/90,40/70,110/170}
		\hgline{-152}{127}{color=blue}
		\hgline{-60}{35}{color=blue}
		\drawBoundaryCircle
	\end{scope}
\end{tikzpicture}
 		\caption{Detecting the intersection of laminations from the
		intersection pattern of the axes of two fundamental group
		elements with each lamination.}
		\label{fig:lam-inc-comb}
	\end{center}
\end{figure}

Another situation in which a pair of fundamental group elements $x,y\in \piS$
detect the intersection of $\lambda$ and $\mu$ occurs when the axes of $x$ and
$y$ cross a common set of leaves of positive measure in both $\tilde{\lambda}$
and $\tilde{\mu}$. In this case, if the axes of $x$ and $y$ cross their common
leaves of $\tilde{\lambda}$ with differing orientations and their common leaves
of $\tilde{\mu}$ with the same orientation, then $\tilde{\lambda}$ and
$\tilde{\mu}$ must intersect, as in~Figure~\ref{fig:lam-inc-or}. In the general
setting of $G$-trees, Definition~\ref{def:co} captures \emph{compatible
orientations} purely in terms of length functions.

\begin{figure}
	\begin{center}
		\begin{tikzpicture}[scale=2.5,font=\small]
	\begin{scope}
		\begin{scope}[shift={(.15,-1.5)},decoration={markings,
			mark = at position 0.5 with {\arrow{>}}}]
			\coordinate (oy) at (-.3,.2);
			\coordinate (ox) at (-.3,0);
			\node[blue, left=1pt of oy] {$y$};
			\draw [blue,postaction={decorate}] (0,.2) -- (oy);
			\node [blue, left=1pt of ox] {$x$};
			\draw [blue,postaction={decorate}] (ox) -- (0,0);
			\draw (-.3,-.1) -- (-.3,-.2) -- (0,-.2) -- (0,-.1);
		\end{scope}
		\draw ([shift={(95:1.1)}]0,0) -- (95:1.15) arc (95:85:1.15) -- (85:1.1);
		\node [blue] at (-152:1.1) {$x$};
		\node [blue] at (-60:1.1) {$y$};
		\node [red] at (-90:1.1) {$\lambda$};
		\setuphypplane
		\lamination{red}{-21/-30,-18/-45,-15/-75,25/50,27/40,-80/85,-80/90,-80/95,-110/95,-105/95,120/150,120/140,130/140,-130/-170,-135/-170,-140/-155}
		\begin{scope}[decoration={markings,
			mark = at position 0.84 with {\arrow{>}},
			}]
			\hgline{-152}{35}{color=blue,postaction={decorate}}
		\end{scope}
		\begin{scope}[decoration={markings,
			mark = at position 0.1 with {\arrow{>}},
			}]
			\hgline{-60}{127}{color=blue,postaction={decorate}}
		\end{scope}
		\drawBoundaryCircle
	\end{scope}
	\begin{scope}[shift={(2.5,0)}]
		\begin{scope}[shift={(-.15,-1.5)},decoration={markings,
			mark = at position 0.5 with {\arrow{>}}}]
			\coordinate (oy) at (.3,.2);
			\coordinate (ox) at (.3,0);
			\node[blue, right=1pt of oy] {$y$};
			\draw [blue,postaction={decorate}] (0,.2) -- (oy);
			\node [blue, right=1pt of ox] {$x$};
			\draw [blue,postaction={decorate}] (0,0) -- (ox);
			\draw (.3,-.1) -- (.3,-.2) -- (0,-.2) -- (0,-.1);
		\end{scope}
		\draw ([shift={(10:1.1)}]0,0) -- (10:1.15) arc (10:-15:1.15) -- (-15:1.1);
		\node [blue] at (-152:1.1) {$x$};
		\node [blue] at (-60:1.1) {$y$};
		\node [red] at (-90:1.1) {$\mu$};
		\setuphypplane
		\lamination{red}{0/198,10/170,10/175,10/178,-15/198,-90/198,-95/215,-95/-105,-15/-50,-20/-40,-35/-40,-50/-90,30/110,40/90,40/70,110/170}
		\begin{scope}[decoration={markings,
			mark = at position 0.84 with {\arrow{>}},
			}]
			\hgline{-152}{35}{color=blue,postaction={decorate}}
		\end{scope}
		\begin{scope}[decoration={markings,
			mark = at position 0.09 with {\arrow{>}},
			}]
			\hgline{-60}{127}{color=blue,postaction={decorate}}
		\end{scope}
		\drawBoundaryCircle
	\end{scope}
\end{tikzpicture}
 		\caption{Detecting the intersection of laminations from the
		intersection orientation of the axes of two fundamental group
		elements.}
		\label{fig:lam-inc-or}
	\end{center}
\end{figure}

The main contribution of this article is that each of these
conditions is a characterization of compatibility.

\begin{theorem}\label{thm:main}
Suppose $A$ and $B$ are irreducible $G$-trees with length functions $\ell$ and
$m$.
The following are equivalent:
\begin{enumerate}
\item\label{thm:horiz} The trees $A$ and $B$ have compatible horizons.
\item The length functions $\ell$ and $m$ have compatible combinatorics.
\item\label{thm:co} The length functions $\ell$ and $m$ are coherently oriented.
\item\label{thm:additive} The sum $\ell + m$ is a length function for an
irreducible $G$-tree.
\item\label{thm:refine} The two trees $A$ and $B$ have a common refinement.
\end{enumerate}
\end{theorem}

\begin{proof}
The equivalence of items \ref{thm:additive} and \ref{thm:refine} is given by
Guirardel and Levitt~\cite{guirardel-levitt}*{Theorem A.10}. We prove the
equivalence of items \ref{thm:horiz}--\ref{thm:co} in Lemma~\ref{lem:compat},
and show the equivalence of \ref{thm:horiz}--\ref{thm:co} to \ref{thm:additive}
in Theorem~\ref{thm:synth-add}.
\end{proof}

These characterizations
have been used in work of the author both to certify incompatibility in the
setting of $F_r$-trees and to extract geometric information from incompatible
trees~\cite{bering}.

The article is organized as follows. Sections~\ref{sec:trees} and
\ref{sec:lfs} recall the theory of $G$-trees and their length functions.
The interplay between the boundary of a $G$-tree and its length function is
elaborated on in Section~\ref{sec:tree-ends}. This is a warm-up for
Section~\ref{sec:comb-compat} which introduces the three new criteria and
proves that they are equivalent to one another. Finally,
Section~\ref{sec:comb-add} proves the equivalence of the three new criteria to
compatibility.

\section*{Acknowledgements}

I thank first and foremost Marc Culler for his guidance in the completion of my
thesis. I am also grateful to my committee, Daniel Groves, Lee Mosher, Peter
Shalen, and Kevin Whyte, for their careful reading of this work and helpful
remarks.
 
\section{Real trees}
\label{sec:trees}

An arc $e$ in a metric space $X$ is the image of an embedding of an interval
$\gamma_e: [a,b]\to X$. In a uniquely geodesic metric space $X$, let $[p,q]$
denote the geodesic from $p$ to $q$. If $p,q,r \in X$ and $r \in [p,q]$, we
will use the notation $[p,r,q]$ for the geodesic path, for emphasis.

\begin{definition}\label{def:real-tree}
	A \emph{real tree} or $\mathbb{R}$-tree $T$ is a connected uniquely
	geodesic metric space such that for any pair of points $p,q\in T$ the
	geodesic $[p,q]$ from $p$ to $q$ is the unique arc from $p$ to $q$. A
	\emph{subtree} of a real tree is a connected subset $S\subseteq T$.
\end{definition}

Throughout this article, when $e \subseteq T$ is an arc in a real tree we will
assume this arc is oriented, that is we have a fixed isometry $\gamma:
[0,\len_T(e)]\to T$ whose image is $e$. We will use the notation $o(e) =
\gamma(0)$ and $t(e) = \gamma(\len_T(e))$ for the origin and terminus of the
arc, and $\bar{e}$ for the reversed orientation. To keep the notation
uncluttered we will not refer to the isometry $\gamma$ unless it is desperately
necessary for clarity. We will always specify an orientation when specifying an
arc (or it will inherit one from context, by being a sub-arc of an oriented
arc).

\begin{definition}
	Let $T$ be a real tree. A point $p\in T$ is a \emph{branch point} if
	$T\setminus\{p\}$ has more than two connected components. The
	\emph{order} of a branch point is the number of connected components of
	$T\setminus\{p\}$. A \emph{direction} based at $p$,
	$\delta_p\subseteq T$, is a connected component of $T\setminus\{p\}$.
\end{definition}

\begin{definition}
	The \emph{visual boundary} of a real tree $T$ based at $p\in T$ is the set
		\[ \partial_p T = \{ \rho \subseteq T |\mbox{ $\rho$ is a
		geodesic ray based at $p$ }\}. \]
	The boundary is topologized by the basis of open sets 
		\[ V(\rho,r) = \{ \gamma\in \partial_pT\,|\,B(p,r)\cap\gamma =
		B(p,r)\cap\rho \} \]
	for $r > 0$ and $\rho \in \partial_pT$. 
\end{definition}

Different base points $p$ give different identifications of the same
boundary~\cite{bridson-haefliger}*{Proposition II.8.8}.

We will write $\partial T$ when the choice of basepoint
is not important and $\omega_T(S) \subseteq \partial T$ for the subset of the
boundary determined by the geodesic rays contained in a subtree $S$. If $S$ is
a bounded subtree, $\omega(S) = \emptyset$.
 
\section{Lengths and actions}
\label{sec:lfs}

\begin{definition}
	Let $G$ be a group and $\rho: G\to Isom(T)$ be an injection, with $T$ a
	real tree, so that  $G$ acts on $T$ on the right.  The triple $(G,\rho,
	T)$ is a \emph{$G$-tree}.
\end{definition}

The action of $G$ will be clear from context and $G$ will be fixed, so we
suppress the notation and refer to a tree $T$ as a $G$-tree. The restriction to
actions where $\rho: G\to Isom(T)$ is injective is not standard in the
literature, some authors allow group actions with kernel; these authors call
actions without kernel \emph{effective}.

We study the geometry of $G$-tree actions via their translation length
functions. The elements of $G$-tree geometry reviewed here are for the most
part based on the exposition given by Culler and Morgan~\cite{culler-morgan},
with other developments cited as relevant.

\begin{definition}
	The \emph{translation length function} of a $G$-tree $T$, denoted
	$\ell_T: G\to \mathbb{R}$ is defined by
		\[ \ell_T(g) = \inf_{p\in T} d_T(p,p\cdot g). \]
\end{definition}

Any $G$-tree $T$ divides the elements
 $g\in G$ into \emph{hyperbolic} elements, when $\ell_T(g) > 0$ and
\emph{elliptic} elements, when $\ell_T(g) = 0$. When an element $g\in G$ is
elliptic, $\Fix(g)$ will denote the set of fixed points of $g$.

\subsection{A taxonomy}

\begin{definition}
	A $G$-tree $T$ is \emph{minimal} if there is no proper $G$ invariant
	subtree $T'\subsetneq T$.
\end{definition}

\begin{definition}
	A $G$-tree $T$ where for all $g\in G$, $\Fix(g)\neq\emptyset$ is \emph{trivial}.
\end{definition}

For finitely generated groups this is equivalent to the condition that $G$ has
a global fixed point, but this is not true for infinitely generated
groups~\cites{morgan-shalen,tits}.

\begin{definition}
	A $G$-tree $T$ is \emph{lineal} if there is a $G$ invariant subtree
	isometric to the line.
\end{definition}

\begin{definition}
	A $G$-tree $T$ is \emph{reducible} if $G$ fixes an end of $T$.
\end{definition}

Lineal and reducible actions are uninteresting from the perspective of
translation length functions.

\begin{theorem}[\cite{culler-morgan}*{Theorem 2.4,2.5}]
	If $T$ is a lineal or reducible $G$-tree, then there is a homomorphism
	$\rho: G\to Isom(\mathbb{R})$ such that $\ell_T(g) = N(\rho(g))$, where
	$N$ is the translation length function of the induced action on
	$\mathbb{R}$.
\end{theorem}

The $G$-trees of interest for this article are the ones whose study is not an
indirect study of subgroups of $Isom(\mathbb{R})$.

\begin{definition}\label{def:irreducible}
	A $G$-tree $T$ is \emph{irreducible} if it is minimal and neither
	trivial, lineal, nor reducible.
\end{definition}

The translation length function is an isometry invariant of irreducible $G$
trees.

\begin{theorem}[\cite{culler-morgan}*{Theorem 3.7}]
	Suppose $A$ and $B$ are two irreducible $G$ trees and $\ell_A =
	\ell_B$. Then there is an equivariant isometry from $A$ to $B$.
\end{theorem}

\subsection{Axes}

\begin{definition}
	The \emph{characteristic set} of some $g\in G$ in a $G$-tree $T$ is the set 
		\[C_g^T = \{ p\in T | d(p,p\cdot g) = \ell_T(g)\} \]
	of points achieving the translation length. When $T$ is clear from
	context we write $C_g$.
\end{definition}

\begin{lemma}[\cite{culler-morgan}*{Lemma 1.3}]\label{lem:cm-axes}
	For any $G$-tree $T$ and $g\in G$, the characteristic set $C_g^T$ is a
	closed non-empty subtree of $T$ invariant under $g$. Moreover,
	\begin{itemize}
		\item If $\ell_T(g) = 0$ then $C_g = \Fix(g)$.
		\item If $\ell_T(g) > 0$ then $C_g$ is isometric to the real
			line and the action of $g$ on $C_g$ is translation by
			$\ell_T(g)$. In this case we call $C_g$ the \emph{axis}
			of $g$.
		\item For any $p\in T$, $d(p,p\cdot g) = \ell_T(g) + d(p,C_g)$.
	\end{itemize}
\end{lemma}

When $g$ is a hyperbolic element of a $G$-tree $T$, $\omega_T(C_g)$ is a pair
of boundary points fixed by $g$.  The action of $g$ on $C_g$ gives $C_g$ a
natural orientation and we always consider an axis oriented by the element
specifying it, so that $C_{g^{-1}}$ is the same set as $C_g$ but with the
opposite orientation. The point of $\partial T$ in the equivalence class of a
positive ray along $C_g$ with the $g$ orientation will be denoted
$\omega_T(g)$. If $g$ is elliptic and $T$ minimal, $\omega_T(C_g) = \emptyset$,
and $\omega_T(g)$ is undefined.

\begin{definition}
	Let $T$ be a $G$-tree. The \emph{$T$-boundary} of $g\in G$,
	$\partial_Tg$ is the empty set if $g$ is elliptic, and the set
	$\{\omega_T(g),\omega_T(g^{-1})\}$ if $g$ is hyperbolic.
\end{definition}

The intersection of characteristic sets is detected by the translation length
function.

\begin{lemma}[\cite{culler-morgan}*{Lemma 1.5}]
	\label{lem:cm-lf-disjoint}
	Let $T$ be a $G$-tree. For any $g,h\in G$ such that $C_g\cap C_h =
	\emptyset$, we have
	\[ \ell(gh) = \ell(gh^{-1}) = \ell(g)+\ell(h)+2d(C_g,C_h) \]
\end{lemma}

This lemma is also used in its contrapositive formulation, if $\ell(gh)\leq
\ell(g)+\ell(h)$, then $C_g\cap C_h \neq \emptyset$. For hyperbolic isometries
there is a more precise relationship between the intersection of characteristic
sets and the length function.

\begin{lemma}[\cite{culler-morgan}*{Lemma 1.8}]
	\label{lem:cm-lf-isect}
	Suppose $g$ and $h$ are hyperbolic in a $G$-tree $T$. Then $C_g\cap C_h
	\neq \emptyset$ if and only if
		\[ \max\{\ell_T(gh),\ell_T(gh^{-1}) \} = \ell_T(g)+\ell_T(h).\]
	Moreover $\ell(gh) > \ell(gh^{-1})$ if and only if $C_g\cap C_h$
	contains an arc and the orientations of $C_g$ and $C_h$ agree on
	$C_g\cap C_h$.
\end{lemma}

These two lemmas are proved by the construction of explicit fundamental
domains. These fundamental domains are sufficiently useful that we detail them
here. That these domains have the claimed properties is a consequence of the proofs of
the previous two lemmas.

\begin{definition}
	Let $T$ be a $G$-tree and suppose $g$ and $h$ are such that $C_g\cap
	C_h = \emptyset$. Let $\alpha = [p,q]$ be the geodesic joining $C_g$ to
	$C_h$. The \emph{Culler-Morgan fundamental domain} for the action of
	$gh$ on $C_{gh}$ is the geodesic 
		\[ [p\cdot g^{-1},p,q,q\cdot h,p\cdot h]. \]
\end{definition}

\begin{definition}
	Let $T$ be a $G$-tree and suppose $g$ and $h$ are such that $C_g\cap
	C_h \neq \emptyset$, at least one of $g$ and $h$ is hyperbolic, and
	that if both $g$ and $h$ are hyperbolic the orientations agree. Let
	$\alpha =[p,q]$ be the possibly degenerate ($p=q$) common arc of
	intersection with the induced orientation. The \emph{Culler-Morgan
	fundamental domain} for the action of $gh$ on $C_{gh}$ is the geodesic
	
		\[ [q\cdot g^{-1},q,q\cdot h]. \]
	If $gh^{-1}$ is also hyperbolic, then the \emph{Culler-Morgan
	fundamental domain} for the action of $gh^{-1}$ on $C_{gh^{-1}}$ is the
	geodesic 
		\[ [q\cdot g^{-1},q\cdot h^{-1} ].\]
\end{definition}

The axes of a minimal $G$-tree provide complete information about the $G$-tree.

\begin{proposition}[\cite{culler-morgan}*{Proposition 3.1}]
	\label{prop:irreducible-union}
	A minimal non-trivial $G$-tree $T$ is equal to the union of the axes of
	the hyperbolic elements.
\end{proposition}

\subsection{Axioms}

For irreducible $G$-trees, length functions provide a complete invariant, as
noted above. Culler and Morgan characterized these length functions in terms of
a list of useful properties; Parry showed that any length function satisfying
these axioms comes from an irreducible $G$-tree~\cite{parry}.

\begin{definition}
	An \emph{axiomatic length function} (or just \emph{length function}) is
	a function $\ell: G\to \mathbb{R}_{\geq 0}$ satisfying the following
	six axioms.
	\begin{enumerate}
		\item $\ell(\id) = 0$.
		\item For all $g\in G$, $\ell(g) = \ell(g^{-1})$.
		\item For all $g,h\in G$, $\ell(g) = \ell(hgh^{-1})$.
		\item For all $g,h\in G$, either 
			\begin{gather*}\ell(gh) = \ell(gh^{-1}) \\
				\mbox{or} \\
				\max\{\ell(gh),\ell(gh^{-1})\} \leq \ell(g)
				+\ell(h).
			\end{gather*}
		\item For all $g,h\in G$ such that $\ell(g) > 0$ and $\ell(h) > 0$,
			either
			\begin{gather*}
				\ell(gh) = \ell(gh^{-1}) > \ell(g)+\ell(h) \\
				\mbox{or}\\
				\max\{\ell(gh),\ell(gh^{-1})\} = \ell(g)+\ell(h).
			\end{gather*}
		\item There exists a pair $g,h\in G$ such that
			\[ 0 < \ell(g) + \ell(h) - \ell(gh^{-1}) < 2\min\{\ell(g),\ell(h)\}. \]
	\end{enumerate}
\end{definition}

\begin{proposition}[\cite{culler-morgan}]
	If $\ell_T$ is the translation length function of an irreducible
	$G$-tree then $\ell_T$ is an axiomatic length function.
\end{proposition}

\begin{theorem}[\cite{parry}]
	If $\ell$ is an axiomatic length function on a group $G$ then there is
	an irreducible $G$-tree $T$ such that $\ell = \ell_T$.
\end{theorem}

In the wider literature, axiom VI is omitted, including the consideration of
all $G$-trees, instead of only irreducible $G$ trees. Without this axiom,
length functions are no longer a complete isometry invariant. A pair of elements witnessing Axiom VI for a given length function $\ell$ is
called a \emph{good pair} for $\ell$.

\subsection{Good pairs}

Culler and Morgan used good pairs in the proof of their uniqueness statement
for $G$-trees coming from a given length function. They give a geometric
definition.

\begin{definition}
	Let $T$ be a $G$-tree. A pair of elements $g,h\in G$ is a \emph{good
	pair} for $T$ if
	\begin{itemize}
		\item the elements $g$ and $h$ are hyperbolic;
		\item the axes $C_g$ and $C_h$ meet in an arc of positive
			length;
		\item the orientations of $C_g$ and $C_h$ agree on the
			intersection;
		\item $\len(C_g\cap C_h) < \min\{\ell(g),\ell(h)\}.$
	\end{itemize}
\end{definition}

\begin{proposition}[\cite{culler-morgan}*{Lemma 3.6}]
	A pair of elements $g,h\in G$ is a good pair for a $G$-tree $T$ if and
	only if $g$ and $h$ witness Axiom VI for $\ell_T$.
\end{proposition}

\begin{lemma}
	\label{lem:cm-good-pair}
	Suppose $g,h \in G$ is a pair of hyperbolic elements of a $G$-tree $T$ whose
	axes intersect in an arc of finite length and the induced orientations
	agree. Then there are integers $A, B > 0$ so that for all $a\geq A$ and
	$b\geq B$, $g^a, h^b$ is a good pair.
\end{lemma}

\begin{proof}
	By hypothesis $g$ and $h$ satisfy the first three points of the
	geometric definition of a good pair. Let $N = \len(C_g\cap C_h)$.
	It is immediate that $A = \lceil N/\ell_T(g) \rceil$ and $B= \lceil
	N/\ell_T(h) \rceil$ are the desired integers.
\end{proof}

The axes of a pair of group elements satisfying the hypotheses of
Lemma~\ref{lem:cm-good-pair} have distinct boundary points; this is a form of
independence seen by the tree, and closely related to the algebraic
independence of group elements in the subgroup generated by a good
pair~\cite{culler-morgan}*{Lemma 2.6}. In the sequel we only need to reference
this boundary independence.

\begin{definition}
	Let $T$ be a $G$-tree. Two hyperbolic elements $g, h\in G$ are
	\emph{$T$-independent} when
	\[ \partial_T g \cap \partial_T h  = \emptyset. \]
\end{definition}
 
\section{Tree ends and length function combinatorics}
\label{sec:tree-ends}

As in the introduction, consider a measured geodesic lamination $\lambda$ of a closed hyperbolic
surface $\Sigma$. Lifting $\lambda$ to the universal cover $\mathbb{H}^2$ gives
a dual $\piS$-tree $T$~\cite{surface-dual-tree}. Corresponding to an oriented
arc $e\subseteq T$ there is a subset of the boundary of $\mathbb{H}^2$. For
each point of $e$ coming from a leaf $\gamma\subseteq \lambda$, $t(e)$
determines a side of $\gamma$ in $\mathbb{H}^2$, and so picks a connected
component of $\mathbb{H}^2\setminus\gamma$. The intersection of the boundaries
of these connected components is the subset of the boundary corresponding to
$e$, as in~Figure~\ref{fig:ebox}. Endpoints of axes of the $\piS$ action on
$\mathbb{H}^2$ are dense in the boundary so this subset can be described
entirely in terms of the group. 

\begin{figure}
	\begin{center}
		\begin{tikzpicture}[scale=3,font=\small,decoration={
    markings,
    mark=at position 0.5 with {\arrow{>}}}
]
	\draw ([shift={(85:1.1)}]0,0) -- (85:1.15) arc (85:-80:1.15) node
	[midway, right, align=center] {seen\\ from $e$} -- (-80:1.1);
	\node at (-.35,0) {$e$} ;
	\draw [postaction={decorate}] (-.3,0) -- (.1,0);
	\setuphypplane
	\lamination{red}{-21/-30,-18/-45,-15/-75,25/50,27/40,-80/85,-80/90,-80/95,-110/95,-105/95,120/150,120/140,130/140,-130/-170,-135/-170,-140/-155}
	\drawBoundaryCircle
\end{tikzpicture}
 		\caption{The part of the boundary ``seen'' from an arc $e$ in the tree dual
		to a lamination.}
		\label{fig:ebox}
	\end{center}
\end{figure}

The description of this subset in terms of the group generalizes to $G$-trees.
Note that for each $p\in e^\circ$, the interior of $e$, the orientation of $e$ picks a unique
direction $\delta_p^e$ based at $p$ such that $t(e)\in \delta_p^e$. The subset
of the boundary of the tree corresponding to $e$ is then
	\[ \bigcap_{p\in e^\circ} \omega_T(\delta_p^e). \]
In the sequel we will be more concerned with describing this directly from the
group.

\begin{definition}\label{def:group-end-dir}
	The \emph{group ends} of a direction $\delta\subseteq T$ is the set of
	group elements
		\[ \delta(G) = \{ g\in G | \omega_T(g)\in\omega_T(\delta) \}. \]
\end{definition}

\begin{definition}\label{def:ebox}
	The \emph{asymptotic horizon} of an oriented arc $e\subseteq T$ of a $G$-tree is
	\[ \ebox{e} = \bigcap_{p\in e^\circ} \delta_p^e(G),\]
	where $\delta_p^e$ is the unique direction based at $p$ such that
	$t(e)\in \delta_p^e$.
\end{definition}

\begin{remark}
	In some figures $\ebox{e}$ will be used to indicate the set
	$\{\omega(g) | g\in \ebox{e}\} \subseteq \partial X$ where $X$ is
	hyperbolic. This abuse of notation is used only in illustrative
	figures, and the set of group elements will play the important role in
	the text.
\end{remark}

The asymptotic horizon of an oriented arc $e$ is all hyperbolic group elements
whose axes have an endpoint visible from $e$, when looking in the forward direction
specified by the orientation. The visibility of group ends is sufficient to
find group elements whose axes either contain $e$ or are disjoint from $e$,
exercises in the calculus of axes that are recorded in the next two lemmas.

To fix notation, for an oriented arc $e\subseteq T$ in a $G$-tree, let $R_e^-$
be the connected component of $T\setminus e^\circ$ containing $o(e)$ and
$R_e^+$ the component containing $t(e)$.

\begin{lemma}\label{lem:arc-axis}
	Suppose $e\subseteq T$ is an oriented arc in a $G$-tree
	$T$. Suppose $g\in \ebox{e}$ and $h\in\ebox{\bar{e}}$. Then there is an
	$N > 0$ such that for all $n\geq N,$ $f=h^{-n}g^n$ is hyperbolic and $e\subseteq
	C_f$. Moreover the orientation of $e$ agrees with the orientation on
	$C_f$ induced by $f$.
\end{lemma}

\begin{proof}
	Consider the intersection $C_g \cap C_h$. There are three cases.

	\emph{Case 1: $C_g\cap C_h = \emptyset$.} Let $a$ be the unique
	shortest oriented arc joining $C_g$ to $C_h$ with $t(a)\in C_g$.
	Take
		\[ N > \frac{d_T(e,a)+\len(e)}{\min\{\ell_T(g),\ell_T(h)\}} \]
	and suppose $n\geq N$. Consider the Culler-Morgan fundamental domain
	for the action of $f=h^{-n}g^n$ on its axis: the geodesic path $b$ in
	$T$ passing through the points
		\[ [o(a)\cdot h^n, o(a), t(a),t(a)\cdot g^n,o(a)\cdot g^n]. \]
	By hypothesis, the axis $C_h$ meets $R_e^-$ in at least a positive
	ray and $hR_e^- \subseteq R_e^-$. If $o(a) \in T\setminus R_e^-$, then
	the ray of $C_h$ based at $o(a)$ directed at $\omega_T(h)$ must pass
	through $o(a)$. By the choice of $N$, $o(a)\cdot h^n \in R_e^-$.
	Similarly, $t(a)\cdot g^n \in R_e^+$. The arc $e$ is the unique
	geodesic in $T$ joining $R_e^-$ to $R_e^+$, hence $e\subseteq b$.
	Moreover, the action of $f$ takes $o(b) = o(a)\cdot h^n$ to 
	$t(b) = o(a)\cdot g^n$, so the orientations of $e$ and $b$ agree, as
	required.

	\emph{Case 2: $C_g\cap C_h = a \neq \emptyset$, $a$ a point or
	arc.} Orient $a$ according to the orientation of $g$. (When $a$ is a
	point, orientation does not matter; we use the convention $o(a) = a =
	t(a)$.) Take
		\[ N > \frac{d_T(e,a) + \len(e) +
		\len(a)}{\min\{\ell_T(g),\ell_T(h)\}} \]
	and suppose $n\geq N$. Again consider the Culler-Morgan fundamental
	domain for the action of $f = h^{-n}g^n$ on its axis. It contains
	(regardless of the agreement between the orientations of $h$ and $a$)
	the geodesic path $b$ in $T$ passing through the points
		\[ [t(a)\cdot h^n,t(a),t(a)\cdot g^n]. \]
	As in the previous case, we find $t(a)\cdot h^n\in R_e^{-}$ and
	$t(a)\cdot g^n\in R_e^+$. We conclude $e\subseteq b$ and the
	orientations agree.

	\emph{Case 3: $C_g\cap C_h$ contains a ray.} If $C_g = C_h$
	then $e\subseteq C_{h^{-1}g} = C_g = C_h$ and $N=1$ suffices. So suppose $C_g\neq
	C_h$. Let $p\in C_g\cap C_h$ be the basepoint of the common ray.
	Take
		\[ N > \frac{d_T(p,e)+\len(e)}{\min\{\ell_T(g),\ell_T(h)\}} \]
	and suppose $n\geq N$. Once more, a fundamental domain for the action
	of $f = h^{-n}g^n$ on its axis can be described. It contains the
	geodesic path $b$ in $T$ passing through the points
		\[ [p\cdot h^n, p, p\cdot g^n].\]
	By the choice of $n$, we find $p\cdot h^n \in R_e^-$ and $p\cdot g^n\in
	R_e^+$. We conclude $e\subseteq b$ and the orientations agree.
\end{proof}

\begin{lemma}\label{lem:pos-axis}
	Suppose $e\subseteq T$ is an oriented arc in a $G$-tree
	$T$. Suppose $g,h \in \ebox{e}$ and $\omega_T(g)\neq
	\omega_T(h)$. Then there is an $N > 0$ such that for all $n\geq N$, $f=h^{-n}g^n$ is
	hyperbolic and $C_f \subseteq R_e^+$.
\end{lemma}

\begin{proof}
	As in the proof of the previous lemma, there are three cases depending
	on $C_g\cap C_h$.

	\emph{Case 1: $C_g\cap C_h = \emptyset$.} Let $a$ be the oriented
	geodesic from $C_h$ to $C_g$, so that $t(a)\in C_g$. Let
	$C_g^{+}$ and $C_h^{+}$ be the positive rays of $C_g$ and $C_h$ based at
	$t(a)$ and $o(a)$ respectively. The infinite geodesic
	$C_h^{+}\cup a \cup C_g^{+}$ has both endpoints in $\partial R_e^+$, so must be contained
	in $R_e^+$, therefore $a\subseteq R_e^+$. At this point it is tempting
	to take $N=1$, however we must exercise care to ensure that the axis of
	the product is contained in $R_e^+$, as this axis is not the infinite
	geodesic previously mentioned.

	Since $g, h \in \ebox{e}$, there is an integer $N_1 > 0$ such that for
	all $n \geq N_1$ we have
	\begin{align*}
		d(t(a)\cdot g^n,t(e)) &> d(t(a),t(e)) \\
		&\mbox{and}\\
		d(o(a)\cdot h^n,t(e)) &> d(o(a),t(e)). 
	\end{align*}

	Let $\alpha_g$ and $\alpha_h$ be the geodesics from $t(e)$ to $C_g$
	and $C_h$ respectively, oriented such that $o(\alpha_g) = o(\alpha_h)
	= o(e)$. Since $g$ acts by translation on its axis in the direction of
	$\omega_T(g)$, there is an $N_2$ such that for all $n\geq N_2$,
	$t(a)\cdot g^n > t(\alpha_g)$ (in the orientation on $C_g$
	induced by the action of $g$). Similarly there is an $N_3$ such that
	for all $n\geq N_3$ $o(a)\cdot h^n > t(\alpha_h)$. Take $N = \max\{N_1,
	N_2, N_3\}$.

	Suppose $n\geq N$. As in the previous lemma, we use the Culler-Morgan
	fundamental domain for the action of $f=h^{-n}g^n$ on $C_f$: the
	geodesic $b$ passing through the points
		\[ [o(a)\cdot h^n,o(a),t(a),t(a)\cdot g^n].\]
	By construction, $b\subseteq R_e^+$. Further, the geodesic from
	$t(\alpha_g)$ to $t(\beta_g)$ is a proper subarc of $b$. Therefore, the
	center $u$ of the geodesic triangle $t(\alpha_g),t(\alpha_h),t(e)$ is in
	the interior of $b$. This point is, by construction, the unique closest point
	of $b$ to $o(e)$. Since $u$ is in the interior of $b$, $u$ is also the
	unique closest point of $C_f$ to $o(e)$, whence $e\nsubseteq C_f$
	and so $C_f\subseteq R_e^+$ as required.

	\emph{Case 2: $C_g \cap C_h = a\neq \emptyset$, $a$ an arc or
	point.} Orient $a$ so that it agrees with the orientation of $C_g$
	induced by the action of $g$ (again with the convention that if $a$ is
	a point, $o(a) = a = t(a)$). If the orientations of $C_g$ and $C_h$
	disagree on $a$, then with $C_g^{+}$ and $C_h^{+}$ defined as in the
	previous case, the previous argument applies. If the orientations of
	$C_g$ and $C_h$ agree on $a$, let $C_g^{+}$ be as before and
	instead take $C_h^{+}$ to be the infinite ray of $C_h$ based at
	$t(a)$. The infinite geodesic $C_g^{+}\cup C_h^{+}$ has both
	endpoints in $\partial R_e^+$, so $t(a) \in R_e^+$. The argument from
	the previous case then applies, \textit{mutatis mutandis}, with $t(a)$
	in place of $o(a)$.

	\emph{Case 3: $C_g \cap C_h$ contains a ray.} In this case, since
	$\omega_T(g)\neq \omega_T(h)$, $C_g\neq C_h$. Let $p$ be the
	basepoint of the common ray $C_g\cap C_h$. Since $g, h\in
	\ebox{e}$, we must have $p\in R_e^+$. The argument from case one then
	applies, \textit{mutatis mutandis}, with $p$ in place of $o(a)$.
\end{proof}
 
\section{Synthetic compatibility conditions}
\label{sec:comb-compat}

Recall the motivating examples of the introduction. The intersection of
boundary sets is naturally captured \emph{in the group} by the asymptotic
horizons, and Figure~\ref{fig:horizon-isect} gives the geometric motivation for
the following definition.

\begin{definition}\label{def:comp-tree}
	Two $G$-trees $A$ and $B$ have \emph{incompatible horizons} if there are
	oriented arcs $a\subseteq A$ and $b\subseteq B$ such that the four sets
		\[ \ebox{a}\cap\ebox{b}\quad
		   \ebox{\bar{a}}\cap\ebox{b}\quad
		   \ebox{a}\cap\ebox{\bar{b}}\quad
		   \ebox{\bar{a}}\cap\ebox{\bar{b}} \]
	are non empty.
\end{definition}

\begin{remark} 
	Behrstock, Bestvina, and Clay~\cite{bbc} consider a similar collection
	of sets when giving a criterion for the presence of a rectangle in the
	Guirardel core of two free simplicial $F_r$ trees.
\end{remark}

Pairs of group elements with either overlapping or disjoint axes for a given
action capture the situations in
Figure~\ref{fig:lam-inc-comb}~and~Figure~\ref{fig:lam-inc-or}. Let $P(G) =
G\times G\setminus \Delta$ be the set of all distinct pairs of elements in our
group.

\begin{definition}\label{def:cd-tree}
	For a $G$-tree $T$ the \emph{overlap set}, $\sO^T \subseteq P(G)$, is
	all pairs $(g,h) \in P(G)$ such that $g$ and $h$ are hyperbolic and
	$C_g\cap C_h$ contains an arc.

	The \emph{disjoint set}, $\sD^T \subseteq P(G)$, is
	all pairs $(g,h)\in P(G)$ such that $C_g\cap C_h =
	\emptyset$.
\end{definition}

This definition can also be stated for length functions.

\begin{definition}\label{def:cd-lf}
	For a length function $\ell:G\to \mathbb{R}_{\geq 0}$ the
	\emph{overlap set}, $\sO^\ell \subseteq P(G)$
	is all pairs $(g,h) \in P(G)$ such that
		\[ \ell(gh)\neq \ell(gh^{-1}). \]

	The \emph{disjoint set}, $\sD^\ell \subseteq P(G)$
	is all pairs $(g,h)\in P(G)$ such that
		\[ \ell(gh) = \ell(gh^{-1}) > \ell(g)+\ell(h). \]
\end{definition}

In the definition for a tree, the hyperbolicity requirement for
membership in $\sO^T$ is necessary, but the length function requirement
implies that $\sO^\ell$ consists of pairs of hyperbolic elements.

\begin{lemma}\label{lem:c-hyp}
	Suppose $\ell$ is a length function on $G$. If $(g,h)\notin \sD^\ell$
	satisfies $\ell(g) = 0$, then
		\[ \ell(gh) = \ell(gh^{-1}) = \ell(h). \]
	In particular all pairs in $\sO^\ell$ are pairs of hyperbolic elements.
\end{lemma}

\begin{proof}
	First, suppose $\ell(h) = 0$ also. Since $(g,h)\notin\sD^\ell$, length
	function axiom IV implies
		\[ \max\{\ell(gh),\ell(gh^{-1})\} \leq \ell(g)+\ell(h) = 0\]
	and we are done. So suppose $\ell(h) > 0$. Let
	$T$ be the irreducible tree realizing $\ell$. It must be the case that
	$C_g\cap C_h$ is non-empty, by Lemma~\ref{lem:cm-lf-disjoint}. Consider
	$p\in T$ and $\alpha$ the shortest arc from $p$ to $C_g\cap C_h$. Let
	$q$ be the endpoint of $\alpha$ in $C_g\cap C_h$.  Since $g$ is
	elliptic, $\alpha\cdot g\cap C_g\cap C_h$ contains $q$, as does
	$\alpha\cap\alpha\cdot g\cap C_h$. The element $h$ is hyperbolic,
	therefore
	\begin{align*}
		d_T(p,p\cdot gh) &\geq d_T(q,q\cdot gh) = d_T(q,q\cdot h) = \ell(h) \\
		d_T(p,p\cdot gh^{-1}) & \geq d_T(q,q\cdot h^{-1}) = \ell(h), 
	\end{align*}
	and we conclude $\ell(gh) = \ell(gh^{-1}) = \ell(h)$ as required.
\end{proof}

\begin{proposition}
	Suppose $T$ is an irreducible $G$-tree with length function $\ell$.
	Then $\sO^T = \sO^\ell$ and $\sD^T = \sO^\ell$, that is,
	definitions~\ref{def:cd-tree}~and~\ref{def:cd-lf} are equivalent.
\end{proposition}

\begin{proof}
	It is immediate from the definitions that $\sO^T \subseteq \sO^\ell$
	and similarly $\sD^T \subseteq \sD^\ell$.

	To demonstrate the reverse inclusions, suppose $(g, h) \in \sO^\ell$.
	By Lemma~\ref{lem:c-hyp}, $g$ and $h$ are hyperbolic. If, for a
	contradiction, $(g,h)\not\in \sO^T$ then either $C_g\cap C_h =
	\emptyset$ or $C_g\cap C_h = \{\ast\}$. In either case we have
		\[ \ell(gh) = \ell(gh^{-1}) = \ell(g)+\ell(h)+d_T(C_g,C_h), \]
	a contradiction.

	If $(g,h)\in \sD^\ell$ but $(g,h)\notin \sD^T$ then $C_g\cap C_h$ is
	non-empty, and so
		\[ \max\{\ell(gh),\ell(gh^{-1})\} \leq \ell(g)+\ell(h), \]
	a contradiction.
\end{proof}

Note that the definitions of $\sO^\ell$ and $\sD^\ell$ depend only on the
projective class of $\ell$; the axis overlap condition is a topological
property of a tree, so this is expected. Also be aware that $\sO^\ell \cup
\sD^\ell \neq P(G)$; pairs such that $\ell(gh) =
\ell(gh^{-1}) = \ell(g)+\ell(h)$ exist.

The interaction of overlap and disjoint sets captures the situations
pictured in Figure~\ref{fig:lam-inc-comb} and Figure~\ref{fig:lam-inc-or}. We state
the definitions in terms of length functions.

\begin{definition}\label{def:cc}
	Two length functions $\ell$ and $m$ on a group $G$ have
	\emph{compatible combinatorics} if 
		\[ \sO^\ell\cap \sD^m = \sD^\ell\cap \sO^m = \emptyset. \]
\end{definition}

\begin{remark} The equivalent definition for trees is vacuous for lineal
	actions. For an lineal action the tree is a line, and the disjoint set
	is empty, hence all lineal actions have compatible combinatorics.
\end{remark}

\begin{definition}\label{def:co}
	Two length functions $\ell$ and $m$ on a group $G$ are \emph{coherently
	oriented} if for all $(g, h)\in \sO^\ell\cap\sO^m$
	\[ \ell(gh^{-1}) < \ell(gh) \Leftrightarrow m(gh^{-1}) < m(gh). \]
\end{definition}

The figures in the motivating discussion strongly suggest that
these three compatibility definitions are equivalent, at least for irreducible
$G$-trees. Further motivation is provided by the following lemma, which
produces pairs of group elements with distinct axes, mirroring the pictures.

\begin{lemma}\label{lem:char-reps}
	Suppose $A$ and $B$ are irreducible $G$-trees that are have incompatible
	horizons. Let $a\subseteq A$ and
	$b\subseteq B$ be arcs witnessing this fact. Then there exist group
	elements $g\in \ebox{a}\cap\ebox{b}$ and
	$\alpha\in\ebox{a}\cap\ebox{\bar{b}}$ such that $C_g^A\cap C_\alpha^A$
	is bounded; and elements $h\in \ebox{\bar{a}}\cap\ebox{\bar{b}}$ and
	$\beta\in\ebox{\bar{a}}\cap\ebox{b}$ such that $C_h^B\cap C_\beta^B$ is
	bounded.
\end{lemma}

\begin{proof} 
	The argument is symmetric so we give the construction of $g$ and
	$\alpha$.  Since $A$ and $B$ are incompatible the relevant sets are
	non-empty.  Take any $g\in \ebox{a}\cap\ebox{b}$ and $\alpha\in
	\ebox{a}\cap\ebox{\bar{b}}$. If $C_g^A\cap C_\alpha^A$ is bounded we are done.
	Suppose $C_g^A\cap C_\alpha^A$ contains a ray. Let $s\in G$ be any
	$A$-hyperbolic element such that $C_s^A \cap C_\alpha^A$ is bounded.
	Such an element exists since $A$ is irreducible (see
	Proposition~\ref{prop:irreducible-union}). If $s$ is elliptic in $B$
	then $\alpha s$ is hyperbolic in both $A$ and $B$ and $C_{\alpha
	s}^A\cap C_\alpha^A$ is bounded, so we may suppose that $s$ is
	hyperbolic in both $A$ and $B$.  Since $g\in\ebox{a}\cap\ebox{b}$ there
	is some $N > 0$ such that $g^Nsg^{-N}\in\ebox{a}\cap\ebox{b}$.  Take
	$g' = g^Nsg^{-N}$. By construction $C_{g'}^A \cap C_\alpha^A$ is
	bounded, so $g',\alpha$ is the desired pair.
\end{proof}

\begin{corollary}\label{cor:char-ends}
	The group elements $g$ and $\alpha$ are $A$-independent, and the group
	elements $h$ and $\beta$ are $B$-independent.
\end{corollary}

For irreducible $G$-trees, the three definitions of compatibility are
equivalent. The strategy suggested by the pictures is to use boundary points to
pick suitable elements of $G$. This philosophy guides the proof below.

\begin{lemma}\label{lem:compat}
	Suppose $\ell$ and $m$ are length functions on $G$ corresponding to the
	irreducible $G$-trees $A$ and $B$ respectively. The following are
	equivalent.
	\begin{enumerate}
		\item The length functions $\ell$ and $m$ \emph{do not} have
			compatible combinatorics.
		\item The length functions $\ell$ and $m$ \emph{are not}
			coherently oriented.
		\item The trees $A$ and $B$ \emph{have} incompatible horizons.
	\end{enumerate}
\end{lemma}

\begin{proof}
	We will show $1 \Leftrightarrow 3$ and $2\Leftrightarrow 3$.

	\emph{($1 \Rightarrow 3$.)} Suppose, without loss of generality, $(g,h)\in
	\sD^\ell \cap \sO^m$. In $A$, by definition $C_g^A\cap C_h^A =
	\emptyset$; let $a\subseteq A$ be the geodesic joining $C_g^A$ and
	$C_h^A$, oriented so that $t(a)\in C_g^A$. We have $g^{\pm}\in
	\ebox{a}$ and $h^\pm\in\ebox{\bar{a}}$. In $B$, again by definition
	there is an arc $b = C_g^B\cap C_h^B$. Without loss of generality we
	assume $g$ and $h$ induce the same orientation on $b$, and use this
	orientation. Then $g, h \in \ebox{b}$ and $g^{-1},h^{-1} \in
	\ebox{\bar{b}}$. We conclude the four sets
		\[ \ebox{a}\cap\ebox{b}\quad
		   \ebox{\bar{a}}\cap\ebox{b}\quad
		   \ebox{a}\cap\ebox{\bar{b}}\quad
		   \ebox{\bar{a}}\cap\ebox{\bar{b}}, \]
	are all non-empty. (See~Figure~\ref{fig:inc-com} for an illustration.)

	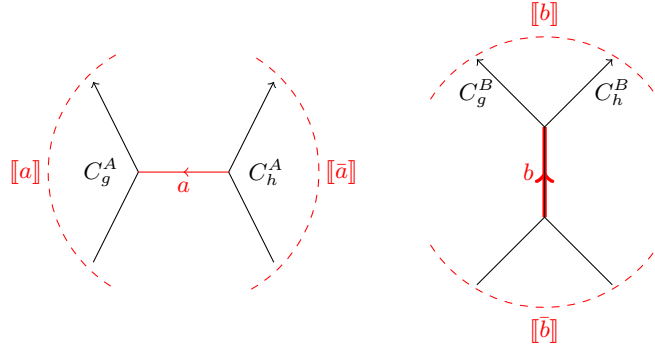
\begin{figure}
		\begin{center}
			\begin{tikzpicture}[scale=.3,font=\small,decoration={
    markings,
    mark=at position 0.5 with {\arrow{>}}}
]
	\begin{scope}[shift={(-8,0)}]
	\coordinate (mg) at (-2,0);
	\coordinate (mh) at (2,0);

	\draw (-4,-4) -- (mg) [->] -- (-4,4);

	\draw (4,-4) -- (mh) [->] -- (4,4);

	\draw [postaction={decorate},red] (mh) -- (mg) node [midway, below]
	{$a$};

	\node [right =4pt of mh] {$C_h^A$};
	\node [left =4pt of mg] {$C_g^A$};

	\draw [dashed, red] ([shift={(120:6)}]0,0) arc (120:240:6);
	\draw [dashed, red] ([shift={(-60:6)}]0,0) arc (-60:60:6);

	\node [red] at (-7,0) {$\ebox{a}$};
	\node [red] at (7,0) {$\ebox{\bar{a}}$};
	\end{scope}
	\begin{scope}[shift={(8,0)}]
	\coordinate (o) at (0,-2);
	\coordinate (t) at (0,2);
	\coordinate (wg) at (-3,5);
	\coordinate (wh) at (3,5);

	\draw [ultra thick, postaction={decorate},red] (o) -- (t) node [midway,
	left] {$b$};

	\draw (-3,-5) -- (o) -- (t) [->] -- (wg);

	\draw (3,-5) -- (o);
	\draw [->] (t) -- (wh);

	\node [below =4pt of wh] {$C_h^B$};
	\node [below =4pt of wg] {$C_g^B$};

	\draw [dashed, red] ([shift={(30:6)}]0,0) arc (30:150:6);
	\draw [dashed, red] ([shift={(-30:6)}]0,0) arc (-30:-150:6);

	\node [red] at (0,7) {$\ebox{b}$};
	\node [red] at (0,-7) {$\ebox{\bar{b}}$};
	\end{scope}

\end{tikzpicture}
 			\caption{Incompatible combinatorics implies
			incompatible trees.}
			\label{fig:inc-com}
		\end{center}
	\end{figure}
	
	\emph{($2 \Rightarrow 3$.)} Let $g, h \in G$ witness the incoherent
	orientation of $\ell$ and $m$, so that $\ell(gh^{-1}) < \ell(gh)$ but
	$m(gh^{-1}) > m(gh)$. Let $a = C_g^A\cap C_h^A$ and $b = C_g^B \cap
	C_h^B$. Since $(g,h)\in \sO^\ell \cap \sO^m$, both $a$ and $b$ are
	arcs. Orient $a$ according to the orientation induced by $g$ on
	$C_g^A$, and similarly orient $b$. The inequality implies that the
	orientation on $a$ induced by $h$ agrees with the orientation on $a$;
	thus $g, h \in \ebox{a}$ and $g^{-1},h^{-1} \in \ebox{\bar{a}}$.
	Similarly, the inequality $m(gh^{-1}) > m(gh)$ implies
	$g,h^{-1}\in\ebox{b}$ and $g^{-1},h \in \ebox{\bar{b}}$. We conclude
	the four sets
		\[ \ebox{a}\cap\ebox{b}\quad
		   \ebox{\bar{a}}\cap\ebox{b}\quad
		   \ebox{a}\cap\ebox{\bar{b}}\quad
		   \ebox{\bar{a}}\cap\ebox{\bar{b}}, \]
	are all non-empty. (See~Figure~\ref{fig:inc-or} for an illustration.)

	\begin{figure}
		\begin{center}
			
\begin{tikzpicture}[scale=.3,font=\small,decoration={
    markings,
    mark=at position 0.5 with {\arrow{>}}}
]
	\begin{scope}[shift={(-8,0)}]
	\coordinate (o) at (0,-2);
	\coordinate (t) at (0,2);
	\coordinate (wg) at (-3,5);
	\coordinate (wh) at (3,5);

	\draw [ultra thick, postaction={decorate},red] (o) -- (t) node [midway,
	left] {$a$};

	\draw (-3,-5) -- (o) -- (t) [->] -- (wg);

	\draw (3,-5) -- (o);
	\draw [->] (t) -- (wh);

	\node [below =4pt of wh] {$C_h^A$};
	\node [below =4pt of wg] {$C_g^A$};

	\draw [dashed, red] ([shift={(30:6)}]0,0) arc (30:150:6);
	\draw [dashed, red] ([shift={(-30:6)}]0,0) arc (-30:-150:6);

	\node [red] at (0,7) {$\ebox{a}$};
	\node [red] at (0,-7) {$\ebox{\bar{a}}$};
	\end{scope}

	\begin{scope}[shift={(8,0)}]
	\coordinate (o) at (0,-2);
	\coordinate (t) at (0,2);
	\coordinate (wg) at (-3,5);
	\coordinate (wh) at (3,-5);

	\draw [ultra thick, postaction={decorate},red] (o) -- (t) node [midway,
	left] {$b$};

	\draw (-3,-5) -- (o) -- (t) [->] -- (wg);

	\draw (3,5) -- (t);
	\draw [->] (o) -- (wh);

	\node [above =4pt of wh] {$C_h^B$};
	\node [below =4pt of wg] {$C_g^B$};

	\draw [dashed, red] ([shift={(30:6)}]0,0) arc (30:150:6);
	\draw [dashed, red] ([shift={(-30:6)}]0,0) arc (-30:-150:6);

	\node [red] at (0,7) {$\ebox{b}$};
	\node [red] at (0,-7) {$\ebox{\bar{b}}$};
	\end{scope}

\end{tikzpicture}
 			\caption{Incoherent orientation implies incompatible
			trees.}
			\label{fig:inc-or}
		\end{center}
	\end{figure}
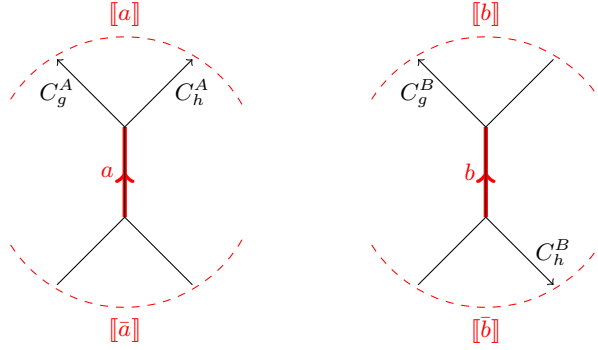

	\emph{($3 \Rightarrow 1 \mbox{ and } 2$.)} Let $a\subseteq A$ and
	$b\subseteq B$ be arcs witnessing the incompatibility of $A$ and $B$.
	Fix group elements $g\in \ebox{a} \cap \ebox{b}, h \in \ebox{\bar{a}}
	\cap \ebox{\bar{b}}, \alpha \in \ebox{a} \cap \ebox{\bar{b}},$ and
	$\beta \in \ebox{\bar{a}} \cap \ebox{b}$ using Lemma
	\ref{lem:char-reps}; by Corollary~\ref{cor:char-ends} the ends of $g$
	and $\alpha$ in $A$ are distinct, and the ends of $h$ and $\beta$ in
	$B$ are distinct.

	Let $N_B$ be the integer guaranteed by Lemma~\ref{lem:arc-axis} applied
	to $g$ and $\alpha$ in $B$, and $N_A$ be the integer supplied by Lemma
	\ref{lem:pos-axis} applied to $g$ and $\alpha$ in $A$. (Note that the
	hypothesis of Lemma~\ref{lem:pos-axis} on the ends of $g$ and $\alpha$ is satisfied.) Set $N =
	\max\{N_A, N_B\}$ and consider $\rho = \alpha^{-N}g^N$. Lemma
	\ref{lem:arc-axis} implies $b \subseteq C_\rho^B$, and Lemma
	\ref{lem:pos-axis} implies $C_\rho^A\subseteq R_a^+$. Choose
	$M$ by a similar process applied to $h$ and $\beta$, so that $\sigma =
	\beta^{-M}h^M$ satisfies $b\subseteq C_\sigma^B$ and $C_\sigma^A
	\subseteq R_a^-$. By construction $C_\rho^B \cap C_\sigma^B \supseteq
	b$, so $(\rho,\sigma)\in \sO^m$; and $C_\rho^A\cap C_\sigma^A =
	\emptyset$, so $(\rho, \sigma)\in \sD^\ell$. Hence $\sD^\ell \cap \sO^m
	\neq \emptyset$ and $\ell$ and $m$ do not have compatible
	combinatorics, as required.

	Continuing the theme, let $J_a$ be the integer given by Lemma
	\ref{lem:arc-axis} applied to $g, h$ and $a\subseteq A$, $J_b$ be the integer
	given by the application to $g, h$ and $b\subseteq B$, and $J = \max\{
		J_a, J_b\}$. Similarly, let $K_a$ be the integer given by Lemma
	\ref{lem:arc-axis} applied to $\alpha, \beta$ and $a$, $K_b$ be the
	integer given by the application to $\alpha,\beta$ and
	$\bar{b}\subseteq B$ (note the reversed orientation), and $K = \max\{
		K_a, K_b\}$. Consider $c=h^{-J}g^J$ and $\gamma =
	\beta^{-K}\alpha^K$. By Lemma~\ref{lem:arc-axis} $a\subseteq C_a^A\cap
	C_\gamma^A$ and all three orientations agree; however $b\subseteq
	C_c^B\cap C_\gamma^B$, but the orientation of $C_c^B$ induced by $c$
	agrees with $b$, while that of $C_\gamma^B$ induced by $\gamma$ agrees
	with $\bar{b}$. Translating this to the length functions $\ell$ and $m$
	we find $(c,\gamma)\in \sO^\ell \cap \sO^m$ and $\ell(c\gamma^{-1}) <
	\ell(c\gamma)$ but $m(c\gamma^{-1}) > m(c\gamma)$, hence $\ell$ and $m$
	are not coherently oriented, as required.
\end{proof}

In light of this lemma a single definition of compatible will be used
throughout the remainder of this article.

\begin{definition}\label{def:compat}
	Two irreducible $G$-trees $A$ and $B$ with length functions $\ell$ and
	$m$ are \emph{synthetically compatible} if, equivalently
	\begin{itemize}
		\item The length functions $\ell$ and $m$ have compatible
			combinatorics.
		\item The length functions $\ell$ and $m$ are coherently
			oriented.
		\item The trees $A$ and $B$ have compatible horizons.
	\end{itemize}
\end{definition}

Note that this definition applies equally well to projective classes of trees.
The first two points depend only on the projective class, so if $\ell$ and $m$
are synthetically compatible then so are $s\ell$ and $tm$ for all $s,t\in
\mathbb{R}_{> 0}$.
 
\section{Synthetic compatibility is equivalent to a common refinement}
\label{sec:comb-add}

\begin{theorem}\label{thm:synth-add}
	Suppose $\ell$ and $m$ are length functions on a group $G$. The sum
	$\ell+m$ is a length function on $G$ if and only if $\ell$ and $m$ are
	synthetically compatible.
\end{theorem}

\begin{proof} First observe that $\ell+m$ always satisfies length function axioms
	I--III. We will focus on IV--VI.

	For the forward implication, suppose $\ell+m$ is a length function on
	$G$. For a contradiction suppose that $\ell$ and $m$ do not have
	coherent orientation, and there is some pair $(g,h)\in \sO^\ell \cap
	\sO^m$ such that $\ell(gh^{-1}) < \ell(gh)$ and $m(gh) < m(gh^{-1})$.
	By Lemma~\ref{lem:c-hyp}, $g$ and $h$ are hyperbolic with respect to
	both $\ell$ and $m$, so both $g$ and $h$ must be hyperbolic in
	$\ell+m$. Length function axiom V implies that for $\ell$ and $m$
	respectively, 
	\begin{align*}
		\ell(gh) &= \ell(g) + \ell(h) \\
		&\mbox{and} \\
		m(gh^{-1}) &= m(g) + m(h).
	\end{align*}
	Taking a sum we have
		\[ \ell(gh) + m(gh^{-1}) = \ell(g) + m(g) + \ell(h) + m(h). \]
	By hypothesis, both
	\begin{align*}
		\ell(gh) + m(gh) &< \ell(gh) + m(gh^{-1}) \\
		\ell(gh^{-1}) + m(gh^{-1}) &< \ell(gh) + m(gh^{-1}).
	\end{align*}
	We conclude that
	\begin{gather*}
		\max\{(\ell + m)(gh), (\ell + m)(gh^{-1})\} < \ell(gh) + m(gh^{-1}) \\
		= (\ell + m)(g) + (\ell + m)(h).
	\end{gather*}
	This is a contradiction, since $\ell + m$ satisfies length function
	axiom V, which implies the above strict inequality must be equality. We
	conclude that $\ell$ and $m$ are compatible.

	For the converse, suppose $\ell$ and $m$ are compatible. As remarked
	previously, $\ell + m$ satisfies length function axioms I--III. We will
	show $\ell + m$ satisfies the remaining axioms.

	\emph{Axiom IV.} Suppose $g, h \in G$. We will proceed through the
	following cases:
	\begin{itemize}
		\item $(g,h) \in \sO^\ell$,
		\item $(g,h) \in \sO^m$,
		\item $(g,h) \in P(G) \setminus (\sO^\ell
			\cup \sO^m).$
	\end{itemize}

	\emph{Case $(g,h)\in \sO^\ell$.} Since $\ell$ satisfies axiom IV, 
		\[ \max\{ \ell(gh),\ell(gh^{-1})\} \leq \ell(g) + \ell(h). \]
	Since $\ell$ and $m$ have compatible combinatorics, $(g,h)\in P(G)
	\setminus \sD^m$, which implies that
		\[ \max\{m(gh),m(gh^{-1}) \} \leq m(g)+m(h). \]
	Hence we may calculate
	\begin{align*}
		\max\{(\ell(gh)+m(gh),\ell(gh^{-1})+m(gh^{-1})\} &\leq
		\max\{\ell(gh),\ell(gh^{-1})\} \\
		&+ \max\{m(gh),m(gh^{-1})\} \\
		&\leq \ell(g)+\ell(h)+m(g)+m(h)
	\end{align*}
	and conclude that in this case $\ell+m$ satisfies axiom IV.

	\emph{Case $(g,h)\in \sO^m$.} The proof is symmetric with the previous
	case.

	\emph{Case $(g,h) \in P(G) \setminus (\sO^\ell \cup \sO^m)$.} In this
	case, by hypothesis both
	\begin{align*}
		\ell(gh) &= \ell(gh^{-1}) \\
		&\mbox{and} \\
		m(gh) &= m(gh^{-1}).
	\end{align*}
	Adding, we conclude
		\[ \ell(gh)+m(gh) = \ell(gh^{-1})+m(gh^{-1}) \]
	as required.

	\emph{Axiom V.} Suppose $g, h\in G$ satisfy $\ell(g) + m(g) > 0$ and
	$\ell(h) + m(h) > 0$. This implies that $g$ and $h$ are both hyperbolic
	in at least one of $\ell$ and $m$. We proceed through the same cases.
	\begin{itemize}
		\item $(g,h) \in \sO^\ell$,
		\item $(g,h) \in \sO^m$,
		\item $(g,h) \in P(G) \setminus (\sO^\ell\cup \sO^m)$.
	\end{itemize}

	\emph{Case $(g,h)\in \sO^\ell$.} In this case, since $\ell$ and $m$ are
	compatible, $(g,h) \notin \sD^m$ and we argue by subcases.
	\begin{itemize}
		\item $m(g) > 0$ and $m(h) > 0$,
		\item $m(g) = 0$ and $m(h) \ge 0$,
		\item $m(g) \ge 0$ and $m(h) = 0$.
	\end{itemize}

	\emph{Subcase $m(g) > 0$ and $m(h) > 0$.} Since $\ell$ and $m$ are
	coherently oriented we have, without loss of generality,
	\begin{align*}
		\ell(gh^{-1}) &< \ell(gh) \\
		&\mbox{and} \\
		m(gh^{-1}) &\le m(gh)
	\end{align*}
	Appealing to axiom V for $\ell$ and $m$ we have,
	\begin{align*}
		\ell(gh) &= \ell(g) + \ell(h) \\
		&\mbox{and} \\
		m(gh) & = m(g) + m(h).
	\end{align*}
	Summing, we conclude
	\begin{align*}
		\ell(gh^{-1}) + m(gh^{-1}) &\leq \ell(gh) + m(gh) \\
		&= \ell(g)+m(g)+\ell(h)+m(h).
	\end{align*}
	Therefore in this subcase $\ell+m$ satisfies axiom V.

	\emph{Subcase $m(g) = 0$ and $m(h)\geq 0$.} By Lemma~\ref{lem:c-hyp},
	$m(gh^{-1}) = m(gh) = m(h)$, so axiom $V$ for $\ell+m$ follows
	immediately from axiom $V$ for $\ell$.

	\emph{Subcase $m(g)\geq 0$ and $m(h) = 0$.} This subcase is symmetric
	with the previous one.

	\emph{Case $(g,h) \in \sO^m$.} This case is symmetric with the previous
	case.

	\emph{Case $(g,h)\in P(G) \setminus (\sO^\ell \cup \sO^m)$.} In this
	case we have
	\begin{align*}
		\ell(gh) = \ell(gh^{-1}) &= \ell(g) + \ell(h) +\Delta_\ell\\
		m(gh) = m(gh^{-1}) &= m(g)+m(h) +\Delta_m
	\end{align*}
	for real numbers $\Delta_\ell,\Delta_m \geq 0$. Immediately we have
	that
		\[ \ell(gh)+m(gh) = \ell(gh^{-1})+m(gh^{-1}) \]
	and from 
		\[ \ell(gh)+m(gh) = \ell(g)+m(g)+\ell(h)+m(h) +
		\Delta_\ell+\Delta_m \]
	we conclude that axiom V is satisfied by $\ell+m$.

	\emph{Axiom VI.} Finally we confirm that $\ell + m$ has a good pair of
	elements. Let $(g, h)$ be a good pair of elements for $\ell$, so that
	\[ 0 < \ell(g)+\ell(h) -\ell(gh^{-1}) < 2\min\{\ell(g),\ell(h)\}. \]

	We check the following cases
	\begin{itemize}
		\item $(g,h) \in \sO^m$,
		\item $(g,h) \notin \sO^m$.
	\end{itemize}

	\emph{Case $(g,h)\in \sO^m$.} In this case, since $\ell$ and $m$ are
	coherently oriented, 
		\[m(gh^{-1}) < m(gh).\]
	By Lemma~\ref{lem:c-hyp},
	$g$ and $h$ are hyperbolic in $m$. Therefore, by Lemma
	\ref{lem:cm-good-pair}, there are positive integers $a$ and $b$ so
	that $(g^a, h^b)$ is a good pair for $m$. Further by Lemma
	\ref{lem:cm-good-pair} the property of being a good pair is preserved
	under taking positive powers, so $(g^a, h^b)$ is a good pair for $\ell$
	also. Adding the good pair inequalities, we calculate
	\begin{align*}
		0 &< \ell(g^a)+m(g^a)+\ell(h^b)+m(h^b) -\ell(g^ah^{-b}) - m(g^ah^{-b}) \\
		&< 2(\min\{\ell(g^a),\ell(h^b)\}+\min\{m(g^a),m(h^b)\}) \\
		&\le 2\min\{\ell(g^a)+ m(g^a),\ell(h^b)+m(h^b)\}.
	\end{align*}
	Hence $(g^a,h^b)$ is a good pair for $\ell + m$.

	\emph{Case $(g,h)\notin\sO^m$.} In this case, since $\ell$ and $m$
	have compatible combinatorics, $(g,h)\notin\sD^m$, and we have
		\[ m(gh^{-1}) = m(gh) = m(g) + m(h). \]
	Adding this to the $\ell$ good pair inequality for $(g,h)$, we have
	\[ 0 < \ell(g)+\ell(h) - \ell(gh^{-1}) =
	\ell(g)+m(g)+\ell(h)+m(h)-\ell(gh^{-1})-m(gh^{-1}) \]
	Since
	\[ 2\min\{\ell(g),\ell(h)\} \leq 2\min\{\ell(g)+m(g),\ell(h)+m(h)\}, \]
	we conclude $(g,h)$ is again a good pair for $\ell + m$.

	This concludes the case analysis. We have verified axioms IV--VI for
	$\ell + m$, and conclude that $\ell + m$ is a length function, as
	required.
\end{proof}


\begin{bibdiv}
	\begin{biblist}
\bib{bbc}{article}{
  author={Behrstock, Jason},
  author={Bestvina, Mladen},
  author={Clay, Matt},
  title={Growth of intersection numbers for free group automorphisms},
  journal={J. Topol.},
  volume={3},
  date={2010},
  number={2},
  pages={280--310},
  issn={1753-8416},
  review={\MR {2651361}},
  doi={10.1112/jtopol/jtq008},
}

\bib{bering}{article}{
  author={Bering, Edgar A., IV},
  title={Uniform independence for Dehn twist automorphisms of a free group},
  journal={Proc. Lond. Math. Soc. (3)},
  volume={118},
  date={2019},
  number={5},
  pages={1115--1152},
  issn={0024-6115},
  review={\MR {3946718}},
  doi={10.1112/plms.12208},
}

\bib{bridson-haefliger}{book}{
  author={Bridson, M. R.},
  author={Haefliger, A.},
  title={Metric spaces of non-positive curvature},
  series={Grundlehren der Mathematischen Wissenschaften [Fundamental Principles of Mathematical Sciences]},
  volume={319},
  publisher={Springer-Verlag, Berlin},
  date={1999},
  pages={xxii+643},
  isbn={3-540-64324-9},
  review={\MR {1744486}},
  doi={10.1007/978-3-662-12494-9},
}

\bib{culler-morgan}{article}{
  author={Culler, Marc},
  author={Morgan, John W.},
  title={Group actions on $\mathbf {R}$-trees},
  journal={Proc. London Math. Soc. (3)},
  volume={55},
  date={1987},
  number={3},
  pages={571--604},
  issn={0024-6115},
  review={\MR {907233}},
  doi={10.1112/plms/s3-55.3.571},
}

\bib{guirardel-core}{article}{
  author={Guirardel, Vincent},
  title={C\oe ur et nombre d'intersection pour les actions de groupes sur les arbres},
  language={French, with English and French summaries},
  journal={Ann. Sci. \'Ecole Norm. Sup. (4)},
  volume={38},
  date={2005},
  number={6},
  pages={847--888},
  issn={0012-9593},
  review={\MR {2216833}},
  doi={10.1016/j.ansens.2005.11.001},
}

\bib{guirardel-levitt}{article}{
  author={Guirardel, Vincent},
  author={Levitt, Gilbert},
  title={JSJ decompositions of groups},
  language={English, with English and French summaries},
  journal={Ast\'{e}risque},
  number={395},
  date={2017},
  pages={vii+165},
  issn={0303-1179},
  isbn={978-2-85629-870-1},
  review={\MR {3758992}},
}

\bib{morgan-shalen}{article}{
  author={Morgan, John W.},
  author={Shalen, Peter B.},
  title={Valuations, trees, and degenerations of hyperbolic structures. I},
  journal={Ann. of Math. (2)},
  volume={120},
  date={1984},
  number={3},
  pages={401--476},
  issn={0003-486X},
  review={\MR {769158}},
  doi={10.2307/1971082},
}

\bib{surface-dual-tree}{article}{
  author={Morgan, John W.},
  author={Shalen, Peter B.},
  title={Free actions of surface groups on ${\bf R}$-trees},
  journal={Topology},
  volume={30},
  date={1991},
  number={2},
  pages={143--154},
  issn={0040-9383},
  review={\MR {1098910}},
  doi={10.1016/0040-9383(91)90002-L},
}

\bib{parry}{article}{
  author={Parry, Walter},
  title={Axioms for translation length functions},
  conference={ title={Arboreal group theory}, address={Berkeley, CA}, date={1988}, },
  book={ series={Math. Sci. Res. Inst. Publ.}, volume={19}, publisher={Springer, New York}, },
  date={1991},
  pages={295--330},
  review={\MR {1105338}},
  doi={10.1007/978-1-4612-3142-4\_11},
}

\bib{tits}{article}{
  author={Tits, J.},
  title={A ``theorem of Lie-Kolchin'' for trees},
  conference={ title={Contributions to algebra (collection of papers dedicated to Ellis Kolchin)}, },
  book={ publisher={Academic Press, New York}, },
  date={1977},
  pages={377--388},
  review={\MR {0578488}},
}

	\end{biblist}
\end{bibdiv}
\end{document}